\documentclass[10pt]{elsarticle}
\usepackage{amssymb,amsmath,amsthm,bm}
\usepackage{amsfonts}
\usepackage{epsfig}
\usepackage{amsbsy}
\usepackage{color}
\usepackage{soul}

\newcommand\R{{\mathbb{R}}}

\def\ri{\mathrm{i}}

\usepackage{graphicx}
\graphicspath{{figures/}}

\newcommand{\realpart}[1]{\operatorname{\sf Re}\!\left(#1\right)}

\newcommand{\beq}{\begin{equation}}
\newcommand{\eeq}{\end{equation}}
\newcommand{\bea}{\begin{eqnarray}}
\newcommand{\eea}{\end{eqnarray}}

\newcommand{\eq}[2]{\begin{equation}\begin{split}#1\end{split}\label{#2}\end{equation}}

\newlength\imagewidth
 \newtheorem{lemma}{Lemma}[section]
 \newtheorem{theorem}{Theorem}[section]
 \newtheorem{corollary}{Corollary}[section]
 \newtheorem{definition}{Definition}[section]
\newtheorem{remark}{Remark}[section]

\renewcommand{\Re}{\operatorname{Re}}

\newcommand\LT{L^2(\mathbb{R})}

\newcommand\HUZ{\widetilde{H}^1}
\newcommand\HT{H^1(\mathbb{R})}

\usepackage{color}
\definecolor{dgreen}{rgb}{0,.6,0}

\usepackage[normalem]{ulem}
\begin{document}

\begin{frontmatter}
\title{Spectral instability of peakons for the $b$-family of Novikov equations}
\author{Xijun Deng}
\ead{xijundeng@yeah.net}
\address{School of Mathematics, Physics and Optoelectronic Engineering,\\
Hubei University of Automotive Technology,
 Shiyan, 442002,  P. R. China }
\address{Department of Mathematics and Physics,\\
Hunan University of Arts and Science,
 Changde, 415000,  P. R. China }

\author{St\'{e}phane Lafortune}
\ead{lafortunes@cofc.edu}
\address{Department of Mathematics,\\
College of Charleston,
 Charleston, SC 29401, USA }

\begin{abstract}
{{In this paper}}, we are concerned with a one-parameter family of peakon equations with cubic nonlinearity parametrized by a parameter usually denoted by the letter $b$. This family is called the ``$b$-Novikov'' since it reduces to the integrable Novikov equation in the case $b=3$.
By extending the corresponding linearized operator defined on functions in $H^1(\mathbb{R})$ to one defined on weaker functions on $L^2(\mathbb{R})$, we prove spectral and linear instability on $\LT$ of peakons in the $b$-Novikov equations for any $b$. We also consider the stability on $\HT$ and show that the peakons are spectrally or linearly stable only in the case $b=3$.
\end{abstract}


\end{frontmatter}

\section{Introduction}
In this paper, we are concerned with the following $b$-family of
Novikov equations (called $b$-Novikov) \cite{MM2013}
\begin{equation}\label{1}
u_t-u_{xxt}+(b+1)u^2u_x=buu_xu_{xx}+u^2u_{xxx}, \,\,\,\,
 t>0,x\in \mathbb{R},
\end{equation}
where $b$ is real parameter. Note that when setting $b=3$, Eq.\eqref{1} becomes the following well-known Novikov equation \cite{Nov09}
\begin{equation}\label{2}
u_t-u_{xxt}+4u^2u_x=3uu_xu_{xx}+u^2u_{xxx},
\end{equation}
which was introduced by Novikov in a symmetry classification of nonlocal partial differential equations with cubic nonlinearity.
The Novikov equation \eqref{2} can be written in terms of the momentum density $m=u-u_{xx}$ as the following evolution form
\begin{equation}\label{3}
m_t+u^2m_{x}+3muu_x=0,
\end{equation}
which can be regarded as a generalization to a cubic nonlinearity of the Camassa-Holm (CH) equation \cite{FF81, CH93}
\begin{equation}\label{4}
m_t+um_{x}+2mu_x=0,
\end{equation}
and the Degasperis-Procesi (DP) equation \cite{DP99}
\begin{equation}\label{5}
m_t+um_{x}+3mu_x=0.
\end{equation}

All the three equations \eqref{3}-\eqref{5} share many common properties. For instance, they are all completely integrable
in the sense that they all have a Lax pair representation, a bi-Hamiltonian structure, and an infinite sequence of conservation laws (see \cite{Con01,Fo95,Fu96}).
Similar to Eqs.\eqref{4} and \eqref{5}, Eq.\eqref{3} also admit the phenomenon of wave breaking, i.e.~the
solution remains bounded but its slope becomes unbounded in finite time, even though it admits initially smooth solutions (see \cite{GL13,FT11,Chen16,Nov09,OR96}).
Another remarkable feature of the Novikov equation \eqref{3}, which is common with the CH and DP equations, is the existence of peaked traveling wave solutions called peakons and it is given by
\begin{equation}\label{6}
u(x,t)=\sqrt{c}e^{-|x-ct|}, x\in \mathbb{R}.
\end{equation}

On the other hand, it is well known that Eqs.\eqref{4} and \eqref{5} can be generalized to the b-family Camassa-Holm equation (called $b$-CH) given by \cite{LP22}
\begin{equation}\label{7}
u_t-u_{xxt}+(b+1)uu_x=bu_xu_{xx}+uu_{xxx},
\end{equation}
which can be derived as the family of asymptotically equivalent shallow water wave equations. Interestingly, the $b$-CH equation \eqref{7} also admits peakon solutions of the form $ce^{-|x-ct|}$ for any $b\in \mathbf{R}$.
By comparison with the $b$-CH equation \eqref{7}, it is easily seen that the $b$-Novikov equation \eqref{1} has nonlinear terms that are cubic, rather than quadratic for the $b$-CH equation. However, like the Novikov equation \eqref{3},
the $b$-Novikov equation \eqref{1} still admits peakon solutions of the form given by \eqref{6} for any $b\in \mathbf{R}$. Moreover, this peakon solution is related to the reformulation of the $b$-Novikov equation \eqref{1} in the weaker
form
\begin{equation}\label{8}
u_t+u^2u_{x}+\frac{1}{2}\phi'\ast(\frac{6-b}{2}uu_x^2+\frac{b}{3}u^3)+\frac{b-2}{4}\phi\ast u_x^3 =0,
\end{equation}
where $\phi(x)=e^{-|x|}$ represents the Green's function for the operator $(1-\partial_x^2)/2$, $(f\ast g)(x):=\int_{\mathbf{R}}f(x-y)g(y)dy$ denotes the convolution operator, and $\phi'$ denotes the piecewise continuous derivative of $\phi$ in $x$.

It follows from \eqref{6} that the exact peakon solution of the $b$-Novikov equation \eqref{1} is of the form $u(x,t)=\sqrt{c}\phi(x-ct)$ and it satisfies the integral equation \eqref{8}.  Notice that the scaling transformation $u(x, t)\rightarrow au(x, a^2t)$ with arbitrary $a\in \mathbf{R}$ leaves \eqref{8} invariant, so we can assume that $c=1$ throughout this paper. It can be easily checked that $\phi(x)=e^{-|x|}$ satisfy the integral equation
\begin{equation}\label{9}
-\phi'+\phi^2\phi'+\frac{1}{2}\phi'\ast(\frac{6-b}{2}\phi\phi'^2+\frac{b}{3}\phi^3)+\frac{b-2}{4}\phi\ast \phi'^3 =0
\end{equation}
piece-wisely on both sides from the peak at $x=0$. Here the integral equation \eqref{9} is the traveling wave equation corresponding to the $b$-Novikov equation \eqref{1}.

It is known that peakons play a distinguished role in the long-time evolution for various nonlinear
dispersive wave equations. Therefore, the stability of peakons is of great interest. In a paper due to Constantin and Strauss \cite{CS00}, they proved that
the single peakons of the CH equation \eqref{4} are orbitally stable by using the conservation property of two energy integrals.  Constantin and Molinet \cite{CM01} introduced
a variational approach to prove the orbital stability of peakons of the CH equation \eqref{4}. It was proven in \cite{JL04, JLE04} that periodic peakons of the CH equation \eqref{4} are also orbitally stable. Lin and Liu \cite{LL09} extended the method in \cite{CS00} and proved that the single peakons for the DP
equation \eqref{5} are orbitally stable in the energy space $L^2(\mathbb{R})$. By constructing a Lyapunov function from the two conserved quantities, Liu et.al \cite{LLQ14} established the orbital stability of peakons
for the Novikov equation \eqref{3}. Recently, various extensions of the $H^1$ orbital stability of peakons have been made. Inspired by the works in \cite{JP16,GP19,GP20}, Natali and Pelinovsky \cite{NP20} showed that for the CH equation \eqref{4}, although the peakons are orbitally stable in $H^1(\mathbb{R})$, they are linearly unstable under $W^{1,\infty}$ perturbations. It was clarified in \cite{MP20} that the perturbations to the peaked periodic wave grow in the $W^{1,\infty}$-norm and may blow up in a finite time in the nonlinear evolution for the CH equation \eqref{4}. Also, the asymptotic stability of peakons in the class of $H^1$ functions was proven in \cite{Mol18}.
By applying the $H^1$ orbital stability and the finite speed propagation property, the $H^1$ asymptotic stability of the peakons for the Novikov equation \eqref{3} was obtained in \cite{Chen21,Pal20,Pal21}.
Moreover, by applying the method of characteristics, Chen and Pelinovsky \cite{CP20} proved that peakon solutions for the Novikov equation \eqref{3} are linearly unstable under $W^{1,\infty}$ perturbations, and the small initial $W^{1,\infty}$ perturbations of peakons can lead to the finite time blow-up of the corresponding solutions. Furthermore, in a paper aimed at clarifying the spectral stability properties of the Novikov peakons, Lafortune \cite{Lafortune2023} proved the spectral instability on $\LT$ and $W^{1,\infty}$, and the spectral stability on $H^1(\mathbb{R})$.

Recently, Lafortune and Pelinovsky \cite{LP22} proved $\LT$ spectral instability of peakons of the b-family of CH equation \eqref{7}, where the instability follows from the presence of the spectrum of a linearized operator in the right half plane of the complex plane. This and the work \cite{Lafortune2023} mentioned above about the spectral stability of the peakons for the Novikov equation, inspired us to study the spectral stability for the peakons of the $b$-Novikov equation \eqref{1}.  In particular, we obtain spectral and linear instability of the peakons on $\LT$ and spectral and linear stability on $\HT$ only in the case $b=3$.

The remainder of this paper is organized as follows. In Section \ref{2}, we first introduce the linearized operator acting on functions in
$H^1(\mathbb{R})$ and then extend
this linearized operator in $L^2(\mathbb{R})$ with a suitable defined domain, after which we formulate the results about linear and spectral instability on $\LT$. The spectrum of the linearized operator on $\LT$ is obtained In Section \ref{Lspec}. 
In Section \ref{HU}, we show that the peakons are spectrally and linearly stable on $\HT$ only if $b=3$. This is done by computing the {{point}} and residual spectrum of the corresponding linear operator.
Section \ref{Con} concludes the paper with a
summary.

\section{Linearized evolution}
\label{2}
We first recall from \cite{Himonas2014} that the $b$-Novikov equation is well-posed on $H^s(\mathbb{R})$ for $s>3/2$.  More precisely, if the initial data belong in
$H^s(\mathbb{R})$ for $s>3/2$,  then there exists a lifespan $T>0$
and a unique solution $u\in  C([0, T],H^s(\mathbb{R}))$ to the $b$-Novikov \eqref{1}. However, the peakon solutions belong to $H^s(\mathbb{R})$ for any $s<3/2$. In that case,
the $b$-Novikov equation was shown to be ill-posed when $b>2$ with the solutions exhibiting norm-inflation, nonuniqueness, or failure of continuity, depending on the values of $b$ and $s$ \cite[Theorem 1]{Himonas2022}.

%
%
In order to study the evolution dynamics near the peaked wave, we proceed as in \cite{CP20} for the Novikov equation and we decompose this weak solution $u(t, x)$ as the sum of a modulated peakon and its perturbation $\upsilon$ in the form:
\begin{equation}\label{10}
u(t,x)=\varphi(x-t-a(t))+\upsilon(t,x-t-a(t)),\,\,\,\,t\in[0,T).
\end{equation}

Substituting \eqref{10} into \eqref{8} and using the integral equation \eqref{9} yields the following linearized evolution equation:
 \begin{eqnarray}\label{11}
& & \upsilon_t = (1+a'(t)-\phi^2)\upsilon_\xi+(a'(t)-2\phi\upsilon)\phi_\xi-\frac{1}{2}\phi'\ast[(6-b)(\phi\phi_\xi\upsilon_\xi+\frac{1}{2}\phi'^2\upsilon) \nonumber\\
& & +b\phi^2\upsilon]-\frac{3(b-2)}{4}\phi\ast(\phi'^2\upsilon_\xi),
\end{eqnarray}
where $\xi:=x-t-a(t)$. Since $\phi_\xi$ is continuous everywhere except at the origin, all the terms of the linearized equation \eqref{11} are continuous at $\xi=0$ if $a'(t)=2v(t,0)$. Furthermore, by using the following two identities established in \cite{CP20}
$$\phi\ast(\phi'^2\upsilon_\xi)=\phi'\ast(\phi^2\upsilon)-\phi\ast[(\phi^2)'\upsilon],$$
$$\phi'\ast(\phi\phi'\upsilon_\xi)=\frac{1}{2}\phi\ast[(\phi^2)'\upsilon]-2\phi'\ast(\phi^2\upsilon)-2\phi'(\phi\upsilon-\upsilon(t,0)),$$
   the above linearized equation \eqref{11} can be simplified and rewritten into the following form:
\begin{equation} \label{12}
\upsilon_t = (1-\phi^2)\upsilon_\xi+(b-4)(\upsilon(t,0)-\phi\upsilon)\phi_\xi+Q(\upsilon),
\end{equation}\\
where
\begin{equation} \label{13}
 Q(\upsilon): =(3-b)[\phi'\ast(\phi^2\upsilon)+\phi\ast(\phi^2\upsilon_\xi)]=(b-3)[\phi\ast (\phi^2)'\upsilon-2\phi'\ast(\phi^2\upsilon)].
\end{equation}

The following lemma characterizes the compactness of the linear operator $Q$.\\

\begin{lemma} \label{2.1Q} The linear operator $Q: L^2(\mathbb{R})\mapsto L^2(\mathbb{R})$ is compact.\end{lemma}

\noindent{\bf Proof.} Each terms in \eqref{13} can be written as the integral operator $\int_{-\infty}^{+\infty}K(\xi,\eta)\upsilon(\eta)d\eta$ for some kernel $K$.
In the case of $\phi\ast(\phi^2)'\upsilon$, the kernel is given by $K=K_1(\xi,\eta)=-2sgn(\eta)e^{-|\xi-\eta|-2|\eta|}$, while for $\phi'\ast \phi^2\upsilon$, we have the kernel $K=K_2(\xi,\eta)=-2sgn(\xi-\eta)e^{-|\xi-\eta|-2|\eta|}$. Note that
$$\int_{-\infty}^{+\infty}\int_{-\infty}^{+\infty}|K_1(\xi,\eta)|^2d\xi d\eta=4\int_{-\infty}^{+\infty}\left(\frac{2}{3}e^{-2|\xi|}-\frac{1}{3}e^{-4|\xi|}\right)d\xi=2$$
and
$$\int_{-\infty}^{+\infty}\int_{-\infty}^{+\infty}|K_2(\xi,\eta)|^2d\xi d\eta=\frac{1}{2}$$
Thus, $K_i\in L^2(\mathbb{R}^2), i=1,2,$ and therefore $Q$ is a Hilbert-Schmidt operator, compact on $L^2(\mathbb{R})$ (see~\cite{Renardy}, p.~262).\qed\\

We will construct a weak version of the linearized equation \eqref{12} that admits solutions in the domain in $\LT$ of the operator
\begin{equation} \label{14}
 L: =(1-\phi^2)\partial_{\xi}+(4-b)\phi\phi'+Q,
\end{equation}
given by
\begin{equation} \label{15}
 {\rm Dom}(L): =\{\upsilon\in L^2(\mathbb{R}): (1-\phi^2)\upsilon'\in L^2(\mathbb{R})\}.
\end{equation}

The following two lemmas describe properties of the linearized operator $L$ that will be useful later.\\

\begin{lemma} \label{2.2}  For $L: {\rm Dom}(L)\subset L^2(\mathbb{R})\mapsto L^2(\mathbb{R})$, we have
\begin{equation} \label{16}
 L(\phi)=(4-b)\phi',\quad  L(\phi')=(b-2)\phi.
\end{equation}
\end{lemma}

\begin{proof} By using direct computation, we have
$$Q(\phi)=(3-b)[\phi'\ast \phi^3+\phi\ast (\phi^2\phi')]=(3-b)(1-\phi^2)\phi',$$
and
$$Q(\phi')=(3-b)[\phi'\ast \phi^2\phi'+\phi\ast (\phi^2\phi'')]=(3-b)\phi(1-\phi^2),$$
for every $\xi\neq 0$. Substituting it into \eqref{14} yields \eqref{16}.\end{proof}

\begin{lemma} \label{SpecSpec}  In the subspace spanned by $\{\phi,\phi'\}$,  $L$ has the two distinct eigenvalues $\pm\sqrt{(b-2)(4-b)}$ if $b\neq 2,4$ and the eigenvalue
0 of multiplicity two if $b=2${ or }$b=4$. In particular, we have that
\begin{equation} \label{SpecSpecr}
 L\text{ has purely imaginary eigenvalues }\pm\ri\sqrt{(b-2)(b-4)}\text{ if $b>4$ or $b<2$}.
\end{equation}
\end{lemma}

\begin{proof} The lemma is proven by applying $L$ to an arbitrary linear combination $\phi$ and $\phi'$ using  \eqref{16}, and requiring the
result to be a multiple of that same linear combination. The fact that 0 is an eigenvalue of multiplicity two if $b=2${ or }$b=4$ is a direct
consequence of the relations given in \eqref{16}.\end{proof}


To construct a weak version of the linearized equation \eqref{12}, we obtain a result valid when the solution is in $\HT$.
Thus, for the lemma below, we consider the domain of $L$ on $\HT$ given by
 \beq
\label{Lop-domain0}
{\mbox{Dom}}(L) = \left\{ v\in H^1(\R): \quad (1-\phi^2) v' \in H^1(\R) \right\}.
\eeq
\\
\begin{lemma} \label{2.3}
Denote
\eq{X:=C(\mathbb{R},H^1(\mathbb{R}))\cap C^1(\mathbb{R}, L^2(\mathbb{R})).}{Xdef}
 with ${\rm Dom}(L)\subset \HT$ given in \eqref{Lop-domain0}.
$\upsilon\in X$ is a solution of the linearized equation \eqref{12}  if and only if $\tilde{\upsilon}:=\upsilon-\upsilon(t,0)\phi$
satisfying $\tilde{\upsilon}(t,0)=0$ is a solution of the linearized equation
\begin{equation}\label{17}
\tilde{\upsilon}_t=L\tilde{\upsilon}-4(b-3)\langle\phi^2\phi',\tilde{\upsilon}\rangle \phi,
\end{equation}
where  $\langle\cdot,\cdot\rangle$ denotes {{the inner product in $L^2(\mathbb{R})$ }}.
\end{lemma}

\begin{proof} In view of \eqref{16}, and substituting $\upsilon(t,\xi)=\tilde{\upsilon}(t,\xi)+\upsilon(t,0)\phi(\xi)$ into \eqref{12} leads to
$$\tilde{\upsilon}_t+\upsilon'(t,0)\phi=L\tilde{\upsilon}+\upsilon(t,0)L\phi+(b-4)\upsilon(t,0)\phi'=L\tilde{\upsilon}.$$
Next, taking the limit $\xi\rightarrow 0$ in \eqref{12} for $\upsilon\in X$ yields
$$\upsilon'(t,0)=\lim_{\xi\rightarrow 0}Q(\upsilon)(\xi)=4(b-3)\langle\phi^2\phi',\upsilon\rangle.$$
Since $\langle\phi^2\phi',\phi\rangle=0$, then $\langle\phi^2\phi',\upsilon\rangle=\langle\phi^2\phi',\tilde{\upsilon}\rangle$. Thus, the two equations above yields \eqref{17}.
Note that $\tilde{\upsilon}(t,\xi)=\upsilon(t,\xi)-\upsilon(t,0)\phi(\xi)$, so $\tilde{\upsilon}(t,0)=0$ since $\phi(0)=1$.
The argument to prove the other direction of the biconditional is very similar except that one substitutes $\tilde{\upsilon}(t,\xi)=\tilde{\upsilon}(t,\xi)-\upsilon(t,0)\phi(\xi)$ into \eqref{17} to obtain \eqref{12}.\end{proof}

Although equation \eqref{17} was obtained for $\tilde{\upsilon}(t,\cdot)\in {\rm Dom}(L) \subset\HT$ given in \eqref{Lop-domain0}, the right-hand side is now well defined for  $\tilde{\upsilon}(t,\cdot)\in {\rm Dom}(L) \subset L^2(\mathbb{R})$ given in \eqref{15}.  We see Equation \eqref{17} as a weak version of the initial linearization \eqref{12}. This is similar to \eqref{8} being a weak version of the $b$-Novikov Equation \eqref{1} that admits peakons as solutions. The weak version of the linearization \eqref{12} given by \eqref{17} enables us to give the following definition of linear stability of the peakons on $\LT$, as done in \cite{LP22} for the $b$-family.

\begin{definition}
\label{d:2.1}
  The peakon solution $u(t,x)=\phi(x-t)$ of the $b$-Novikov equation \eqref{8} is said to be linearly stable on $\LT$ if, for every $\tilde{\upsilon}_0\in {\rm Dom}(L) \subset L^2(\mathbb{R})$, there exists a positive constant $C$ and a unique solution $\tilde{\upsilon}\in C(\mathbb{R},{\rm Dom}(L))$ given in \eqref{15} to the linearized equation \eqref{17}
with $\tilde{\upsilon}(0,\xi)=\tilde{\upsilon}_0(\xi)$ such that $||\tilde{\upsilon}(t,\cdot)||_{L^2}\leq C||\tilde{\upsilon}_0||_{L^2},t>0$.  On the contrary, it is said to be linearly unstable.
\end{definition}

We now prove that the peakons are linearly {{unstable}} on $\LT$. To do so, in  Lemma \ref{2.4} below, we rewrite the linearized equation \eqref{17} in such a way that any unstable spectrum of $L$ create a linear instability in the sense of Definition \ref{d:2.1}. We then state Theorem \ref{2.1} giving the spectrum of $L$, which we prove in the next section. Lemma \ref{2.4} and Theorem \ref{2.1} are then used to prove the linear instability of the peakons for any value of $b$.\\

\begin{lemma} \label{2.4} Denote
 \eq{Y:=C(\mathbb{R},{\rm Dom}(L))\cap C^1(\mathbb{R},L^2(\mathbb{R})),}{Ydef}
 with ${\rm Dom}(L)\subset \LT$ given in \eqref{15}. Then, $\tilde{\upsilon}\in Y$ is a solution to the linearized equation \eqref{17} if $\omega:=\tilde{\upsilon}-\alpha \phi-\beta\phi'$ is
a solution of the linearized equation
\begin{equation}\label{18}
 \frac{d\omega}{dt}=L\omega,
\end{equation}
with $\alpha$ and $\beta$ satisfying the system
\begin{equation}\left\{\begin{array}{l}\label{19}
\frac{d\alpha}{dt}=(4-b)\beta-4(b-3)\langle\phi^2\phi',\omega\rangle,\\ \\
\frac{d\beta}{dt}=(4-b)\alpha.
\end{array}\right.\end{equation}\\
Consequently, the peakon is linearly unstable in the sense of Definition 2.1
for every $b\neq 4$.
\end{lemma}

\begin{proof} Substituting
\eq{\tilde{\upsilon}=\omega+\alpha \phi+\beta\phi'}{vtt}
  into \eqref{17} yields
$$\omega_t+\alpha'(t)\phi+\beta'(t)\phi'=L\omega+(4-b)\alpha\phi'+(4-b)\beta\phi-4(b-3)\langle\phi^2\phi',\omega\rangle\phi,$$
where we have used $\langle\phi^2\phi',\phi\rangle=0, \langle\phi^2\phi',\phi'\rangle=\frac{1}{2}$ and the identities \eqref{16}.
Separating $\phi, \phi', \omega$ yields \eqref{18} and \eqref{19}.
\end{proof}

It follows from Lemma \ref{2.3} and Lemma \ref{2.4} that the linearized
evolution equation \eqref{11} defined in $X$ given in \eqref{Xdef} can be reduced to the linearized evolution equation \eqref{18} defined in $Y$, given in \eqref{Ydef}.
Therefore,  the linear stability
of the peakons is determined by the spectral properties of the operator $L$ in \eqref{14}-\eqref{15}, where the operator $L$ is defined
according to the standard definition described in \cite{BS18}.

The next theorem, proven in Section \ref{Lspec}, gives the spectrum of $L$.\\

\begin{theorem}\label{2.1} The spectrum of the linear operator $L$ defined by \eqref{14}-\eqref{15} is given
by
$$\sigma(L):=\{\lambda\in \mathbb{C}:|\Re (\lambda)|\leq|5-b|\}.$$
If $b<5$, the point spectrum contains the region defined by $0 < |\realpart{\lambda}|<5-b$.
There is residual spectrum if $b>5$, located in $0 < |\realpart{\lambda}|<b-5$.
 The continuous
spectrum is located in $\realpart{\lambda} = 0$ and $\realpart{\lambda} = |5-b|$, except the eigenvalues given in \eqref{SpecSpecr}
embedded into the continuous spectrum for $b\leq 2$ or $b\geq 4$. \\
\end{theorem}

\begin{corollary}\label{StabRes1} The peakon is linearly unstable on $L^2(\mathbb{R})$ in the sense of Definition 2.1 for every $b\neq 5$.\end{corollary}

\begin{proof} The corollary is proven by using a similar method as in \cite{LP22} for the $b$-family peakon. Namely, we use the formulation of the linearized equation given by the system \eqref{18} and \eqref{19}
to prove linear instability in the cases $b\neq 4$ and $b\neq 5$ separately. We first prove that the solution is unstable
for every $b\neq 4$ by constructing a solution  to \eqref{19} that grows exponentially in time. In the case $b\neq 5$, we use the spectrum of $L$ given by Theorem \ref{2.1}
to construct a solution unbounded in time to equation \eqref{18}.

{\noindent}{\em Instability when $b\neq 4$.}  Setting $\omega=0$ in \eqref{18} and \eqref{19} gives the second-order homogeneous system
\begin{equation*}\left\{\begin{array}{l}
\frac{d\alpha}{dt}=(4-b)\beta,\\ \\
\frac{d\beta}{dt}=(4-b)\alpha.
\end{array}\right.\end{equation*}\\
This implies that for every $b\neq 4$, there exists a solution to the linearized equation \eqref{17}, obtained by the transformation \eqref{vtt}, with exponential growth $e^{|b-4|t}$.

{\noindent}{\em Instability when $b\neq 5$.} The instability result when $b\neq 5$ is obtained by using the results of Theorem \ref{2.1}  to construct solutions to the equation \eqref{18} that grow exponentially in time.

  If $b<5$, then any $\lambda_0$ such that $\realpart{\lambda_0} \in (0,5-b)$ is a eigenvalue for the point spectrum of $L$ for some corresponding eigenfunction  $w_0$. Choose $\lambda_0$ different from the eigenvalues given by Lemma \ref{SpecSpec}. Then
the linearized equation (\ref{18}) has the exact solution
$\omega(t,\xi) = e^{\lambda_0 t} w_0(\xi)$, with $w_0$ not in the subspace spanned by $\{\phi,\phi'\}$. This gives rise to a solution to the linearized equation \eqref{17} obtained through the transformation \eqref{vtt} with exponential time growth of
$\| \tilde{v}(t,\cdot) \|_{L^2}$. 	

If $b>5$, then any  $\lambda_0 \in (0,b-5)$ is a real eigenvalue for the residual spectrum of $L$, then $\lambda_0$ is an eigenvalue for the point spectrum of $L^*$ since $\sigma_r(L) = \sigma_p(L^*)$ by Lemma 6.2.6 in \cite{BS18}. The eigenfunction of $L^*$ for the eigenvalue $\lambda_0$ is in $\LT$.
We consider the decomposition
$\omega(t,\xi) = a(t) w_0(\xi) + \tilde{w}(t,\xi)$, where $a(t)$ is uniquely determined by the orthogonality condition $\langle w_0, \tilde{w}(t,\cdot) \rangle = 0$. Both $a(t)$ and $\tilde{w}(t,\xi)$ are found from
$$
\frac{da}{dt} w_0 + \frac{d \tilde{w}}{dt} = a L w_0 + L \tilde{w}.
$$
Projecting to $w_0$ yields $\frac{da}{dt} = \lambda_0 a$ with
the exponential growth of $a(t)$ if $\lambda_0 > 0$.

In both cases ($b<5$ and $b>5$), the linear evolution of $\tilde{v} \in Y$ solution of \eqref{17} grows exponentially
in the $\LT$ norm. By the definition of linear stability, the peakons are linearly unstable in $\LT$.

\end{proof}

\section{Spectrum of $L$ on $\LT$.}
\label{Lspec}

In this section, we study the spectrum of $L$ given by \eqref{14}-\eqref{15}. The operator $L$ can be decomposed as $L=L_0+Q$,  where
$L_0: {\rm Dom}(L)\subset L^2(\mathbb{R})\mapsto L^2(\mathbb{R})$ is given by
\begin{equation}\label{20}
L_0:=(1-\phi^2)\partial_\xi+(4-b)\phi\phi',
\end{equation}
and $Q:L^2(\mathbb{R})\mapsto L^2(\mathbb{R})$ is the compact operator defined by Lemma \ref{2.1Q}.
 According to Theorem 1 in \cite{GP20},
if $\sigma_p(L_0)\cap \rho(L)=\sigma_p(L)\cap \rho(L_0)=\emptyset$, then $\sigma(L)=\sigma(L_0)$. First, we compute the spectrum of $L_0$.

\subsection{Spectrum of $L_0$.}
The following theorem gives the spectrum of  $L_0$.

\begin{theorem} \label{3.1} The spectrum of $L_0:{\rm Dom}(L)\subset L^2(\mathbb{R})\mapsto L^2(\mathbb{R})$ defined by \eqref{20} is given by
  $$\sigma(L_0):=\{\lambda\in \mathbb{C}:|\Re (\lambda)|\leq|5-b|\}.$$
If $b<5$, the point spectrum is in the region defined by $0 < | \realpart{\lambda}|<5-b$.
 If $b>5$ the
residual spectrum is given by $ 0<| \realpart{\lambda}| < b -5$. The continuous
spectrum is located in $\realpart{\lambda} = 0$  and $\realpart{\lambda} = \pm|5-b|$.
\end{theorem}

\begin{proof} Firstly, we compute the point spectrum of $L_0$. We consider the differential equation $L_0 \upsilon=\lambda \upsilon$, that is
\begin{equation}\label{21}
(1-\phi^2)\frac{d \upsilon}{d\xi}+(4-b)\phi\phi'\upsilon=\lambda \upsilon,\quad \xi \in \mathbb{R}.
\end{equation}
It is easy to see that Eq.\eqref{21} has the following general solution:
\begin{equation}\label{22}
 \upsilon(\xi)=
 \begin{cases}
  \frac{\upsilon_{+}(e^{2\xi}-1)^{\frac{\lambda}{2}}}{(1-e^{-2\xi})^{\frac{b}{2}-2}}, & \xi>0,\\
  \frac{\upsilon_{-}e^{\lambda\xi}}{(1-e^{2\xi})^{\frac{\lambda+b}{2}-2}}, & \xi<0,
 \end{cases}
\end{equation}
where $\upsilon_{+}$ and $\upsilon_{-}$ are arbitrary constants. Notice that for the above differential equation \eqref{21} if $\lambda=\lambda_0$ is an eigenvalue with the eigenfunction $\upsilon=\upsilon_0(\xi)$, then $\lambda=-\lambda_0$ is an eigenvalue with the
eigenfunction $\upsilon=\upsilon_0(-\xi)$.  This implies that one only need to consider the case of $\realpart{\lambda}\geq 0$.

From \eqref{22}, $\upsilon(\xi)\sim \upsilon_{+}e^{\lambda\xi}$ as $\xi\rightarrow +\infty$, then $\upsilon\in L^2(\mathbb{R})$  for $\realpart{\lambda}\geq 0$ only if $\upsilon_{+}=0$.
Furthermore, since $\upsilon(\xi)\sim |\xi|^{2-\frac{b+\lambda}{2}}$ as $\xi\rightarrow 0^{-}$, then $\upsilon\in L^2(\mathbb{R})$  only if $\frac{\displaystyle{\realpart{\lambda}}}{2}+\frac{b}{2}-2<\frac{1}{2}$, i.e. $\realpart{\lambda}<5-b$.
Summarizing and using the symmetry
above, we can conclude that the point spectrum $\sigma_p(L_0)$ is located in $0 < | \realpart{\lambda}|<5-b$ if $b<5$.

By Lemma 6.2.6 in \cite{BS18}, if $\sigma_p(L_0)=\emptyset$, then $\sigma_r(L_0)=\sigma_p(L_0^{\ast})$, where
$$ L_0^{\ast}=-(1-\phi^2)\partial_{\xi}+(6-b)\phi\phi'$$
is the adjoint operator to $L_0$ in $L^2(\mathbb{R})$. Consider the differential equation $L_0^{\ast} \upsilon=\lambda \upsilon$, that is
\begin{equation}\label{23}
-(1-\phi^2)\frac{d \upsilon}{d\xi}+(6-b)\phi\phi'\upsilon=\lambda \upsilon,\quad \xi \in \mathbb{R}.
\end{equation}
We can easily obtain the following general solution to \eqref{23}
\begin{equation}\label{24}
 \upsilon(\xi)=
 \begin{cases}
  \upsilon_{+}e^{-\lambda\xi}(1-e^{-2\xi})^{\frac{b-\lambda}{2}-3} , & \xi>0,\\
  \upsilon_{-}e^{-\lambda\xi}(1-e^{2\xi})^{\frac{b+\lambda}{2}-3} , & \xi<0,
 \end{cases}
\end{equation}
where $\upsilon_{+}$ and $\upsilon_{-}$ are arbitrary constants. Proceeding similarly with the arguments above, we find that  \eqref{24} defines a function in
$L^2(\mathbb{R})$ if and only if $\upsilon_{-}=0, b>5$ and $0 < | \realpart{\lambda}|<b-5$. This gives that the residual spectrum  $\sigma_r(L_0)$ is located in $0 < | \realpart{\lambda}|<b-5$ if $b>5$.

Next, we obtain a sufficient condition for a value of $\lambda$ to be in the resolvent set of $L_0$.  Consider the resolvent equation
\begin{equation}\label{25}
L_0\upsilon-\lambda\upsilon=f,\ \ f\in L^2(\mathbb{R}),
\end{equation}
where we assume that $\realpart{\lambda}\geq 0$ without loss of generality. In view of the definition of $L_0$, multiplying both sides of \eqref{25} by $\bar{\upsilon}$ and integrating over $\mathbb{R}$ leads to
\begin{equation}\label{26}
<((1-\phi^2)\upsilon)',\upsilon>+(6-b)<\phi\phi'\upsilon,\upsilon>-\lambda \parallel\upsilon\parallel^2=<f,\upsilon>.
\end{equation}
Note that $\lim_{\xi\rightarrow \pm \infty}\upsilon(\xi)=0$ for $\upsilon\in {\rm Dom}(L)$, and by integration by parts, we obtain
$$
<((1-\phi^2)\upsilon)',\upsilon>=-<\upsilon,((1-\phi^2)\upsilon)'>-2<\phi\phi'\upsilon,\upsilon>,
$$
which means that
\begin{equation}\notag
\Re(<((1-\phi^2)\upsilon)',\upsilon>) =-<\phi\phi'\upsilon,\upsilon>.
\end{equation}
Taking the real part of \eqref{26}, we have
\begin{equation}\label{28}
\realpart{\lambda}\|\upsilon\|^2+(b-5)<\phi\phi'\upsilon,\upsilon> =-\Re(<f,\upsilon>).
\end{equation}
Note that $-\|\upsilon\|^2\leq<\phi'\phi\upsilon,\upsilon>\leq \|\upsilon\|^2$, and using the Cauchy-Schwarz inequality in \eqref{28} in the case $b\leq 5$, we have
$$
(\realpart{\lambda}+b-5)\|\upsilon\|^2\leq|\Re(<f,\upsilon>)|\leq\|f\|\|\upsilon\|.
$$
Thus, for every $\realpart{\lambda}>5-b$, there exists $C_{\lambda}$ such that $\|\upsilon\|\leq C_{\lambda}\|f\|$. This implies that $\lambda\in \rho(L_0)$.
Similarly, in the case $b\geq 5$, we have
$$
(\realpart{\lambda}-b+5)\|\upsilon\|^2\leq|\Re(<f,\upsilon>)|\leq\|f\|\|\upsilon\|,
$$
hence $\realpart{\lambda}>b-5$ belongs to $\rho(L_0)$.

We computed the point and residual spectra and found that their union is given by the region defined by $0<|\realpart{\lambda}<|b-5|$. We also proved that the resolvent set includes the region defined by $\realpart{\lambda}>|b-5|$. As a consequence of the fact that the spectrum is closed, the continuous spectrum is located in $\realpart{\lambda} = 0$  and $\realpart{\lambda} = \pm|5-b|$.

\end{proof}

\subsection{Point spectrum of $L$.}
We consider the spectral problem
\begin{equation}\label{29}
L\upsilon-\lambda\upsilon=0,\ \ \upsilon\in {\rm Dom} (L)\subset L^2(\mathbb{R}).
\end{equation}
By Lemma \ref{SpecSpec}, we have eigenvalues on the imaginary axis given in \eqref{SpecSpecr}. In particular we have that $0$ is a double eigenvalue if $b=2$ or $b=4$.
Now, we shall look for other solutions of the spectral problem \eqref{29}. The following lemma characterizes the point spectrum of $L$.

\begin{lemma}\label{3.1l} In addition to the eigenvalues on the imaginary axis given in \eqref{SpecSpecr}, the linear operator $L:{\rm Dom}(L)\subset L^2(\mathbb{R})\rightarrow L^2(\mathbb{R})$  defined by \eqref{14} admits the point
spectrum for $0 <|\realpart{\lambda}|< 5-b$ if $b<5$.
\end{lemma}

\noindent{\bf Proof.} It is easy to see that the spectral problem \eqref{29} has the same symmetry as the differential equation
\eqref{21}, so it is sufficient to consider
the case $\realpart{\lambda}\geq 0$.

Applying the operator $1-\partial_\xi^2$ to \eqref{29} separately for $\xi< 0$ and $\xi>0$ yields the
following differential equation
\begin{equation}\notag
\lambda(\upsilon-\upsilon'')=(1-\phi^2)(\upsilon'-\upsilon''') -b\phi\phi'(\upsilon-\upsilon'').
\end{equation}

Let $m:=\upsilon-\upsilon''$, the differential equation \eqref{29} becomes the first-order equation
\begin{equation}\notag
 (1-\phi^2)\frac{dm}{d\xi} -b\phi\phi'm=\lambda m,
\end{equation}
which admits the exact solution in the form
\begin{equation}\label{32}
 m(\xi)=\upsilon-\upsilon''=
 \begin{cases}
  m_{+}e^{\lambda\xi}(1-e^{-2\xi})^{\frac{\lambda-b}{2}} , & \xi>0,\\
  m_{-}e^{\lambda\xi}(1-e^{2\xi})^{-\frac{b+\lambda}{2}} , & \xi<0,
 \end{cases}
\end{equation}
where $m^+$ and $m^-$ are arbitrary constants. If $\upsilon\in L^2(\mathbb{R})$, then $m\in H^{-2}(\mathbb{R})$. Since $m(\xi)\sim m_+e^{\lambda\xi}$ as $\xi\rightarrow +\infty$,
then $\upsilon\in L^2(\mathbb{R})$ exists for $\realpart{\lambda}\geq 0$ if and only if $m_+=0$. This means that $\upsilon(\xi)=c_+e^{-\xi}$, where $c_+$ is arbitrary constant. Similarly, we have
$m(\xi)\sim m_-e^{\lambda\xi}$ as $\xi\rightarrow -\infty$. If $\realpart{\lambda}=0$, then $\upsilon\in L^2(\mathbb{R})$ exists if and only if $m_-=0$, namely $\upsilon(\xi)=c_-e^{\xi}$, where $c_-$ is arbitrary constant.
Thus, $\upsilon(\xi)=c_1\phi+c_2\phi'$ with $c_1\pm c_2=c_{\mp}$, which corresponds to the eigenvalues on the imaginary axis given in \eqref{SpecSpecr}, in particular the double eigenvalue $\lambda=0$ if $b=2$ or $b=4$.

Next, we only need to consider the case $m_+=0, m_-\neq 0$ and $\realpart{\lambda}>0$. With the
normalization $m_-=1$, the solution of \eqref{32} for $\xi<0$ can be written in the form
\begin{equation}\label{33}
\upsilon(\xi)=e^{\lambda\xi}(1-e^{2\xi})^{2-\frac{\lambda+b}{2}}f(\xi),\quad \xi<0,
\end{equation}
where $f(\xi)$ satisfies the following second-order differential equation
 \begin{eqnarray}\label{34}
& &  (1-e^{2\xi})^2f''+2(1-e^{2\xi})(\lambda+(b-4)e^{2\xi})f' \nonumber\\
& & +[(b-5)(b-3)e^{4\xi}+2(b-3)(\lambda+1)e^{2\xi}+\lambda^2-1]f=-1.
\end{eqnarray}
The homogeneous part of the above equation with the regular singular point $\xi=0$
is associated with the indicial equation
$$ 4\sigma^2-4(\lambda+b-3)\sigma+(\lambda+b-4)(\lambda+b-2)=0$$
for $f(\xi)\sim \xi^\sigma$. If $\lambda+b\neq \{2,4\}$, then $0$ is not a root
of the indicial equation, while $\lambda+b=\{2,4\}$, then $0$ is a simple root of the
indicial equation. By the Frobenius theory \cite{TE12}, the differential equation \eqref{34} has a particular solution
with the following behavior near the regular singular point $\xi=0$ as $\xi\rightarrow 0^-$
\begin{equation}\notag
 f(\xi)\sim
 \begin{cases}
   1+\mathcal{O}(|\xi|), & \lambda+b\neq \{2,4\},\\
   \log|\xi|+\mathcal{O}(|\xi|\log|\xi|), & \lambda+b= \{2,4\},
 \end{cases}
\end{equation}
which leads to the corresponding behavior of $\upsilon(\xi)$ as $\xi\rightarrow 0^-$ from \eqref{33}
\begin{equation}\label{36}
 \upsilon(\xi)\sim
 \begin{cases}
   |\xi|^{2-\frac{\lambda+b}{2}}, & \lambda+b\neq \{2,4\},\\
   |\xi|\log|\xi|, & \lambda+b=2,\\
   \log|\xi|,& \lambda+b=4.
 \end{cases}
\end{equation}
Therefore, \eqref{32} defines a function $\upsilon\in L^2(\mathbb{R})$ if and only if
\eq{
2-\frac{\lambda+b}{2}>-\frac{1}{2}\text{ and }\realpart{\lambda}>0,}{condL2}
which implies  $0<\realpart{\lambda}<5-b$ for $b<5$.
Summarizing and noting the symmetry above, the point spectrum $\sigma_p(L)$ exists if $b<5$ and it is located at $0<|\realpart{\lambda}|<5-b$.\qed\\

\subsection{Residual spectrum of $L$.}
We consider the spectral problem
\begin{equation}\notag
L^*\upsilon=\lambda\upsilon,\;\;\;\upsilon\in {\rm Dom}(L)\subset L^2(\mathbb{R}),
\end{equation}
where the adjoint operator $L^*: {\rm Dom}(L)\subset L^2(\mathbb{R})\mapsto L^2(\mathbb{R})$ is defined by
\begin{equation}\label{38}
L^*\upsilon=(\phi^2-1)\upsilon_\xi+(6-b)\phi\phi'\upsilon+2(b-3)[\phi\phi'(\phi\ast\upsilon)+\phi^2(\phi'\ast \upsilon)].
\end{equation}
By Lemma 6.2.6 in \cite{BS18}, $\sigma_r(L)\subset \sigma_p(L^*)$. The following lemma describes the point spectrum of $L^*$.\\

\begin{lemma} \label{3.2} The linear operator $L^*$ defined by
\eqref{38} has a nonempty point spectrum if $b>5$, in which case it is located for
$0<|\realpart{\lambda}|<b-5$.
\end{lemma}

\noindent{\bf Proof.} It is sufficient to consider the case of $\realpart{\lambda}\geq 0$. By making the substitution $\upsilon=k-k''$ and assuming that $k$ and $k'$ are bounded and continuous functions, we can obtain
\begin{equation}\label{39}
(1-\phi^2)k_{\xi\xi\xi}+(b-6)\phi\phi'k_{\xi\xi}+[(4b-11)\phi^2-1]k_{\xi}+(3b-6)\phi\phi'k=\lambda(k-k_{\xi\xi}),
\end{equation}
which can be factored out as follows (for both cases $\xi>0$ and $\xi<0$)
\begin{equation}\label{40}
(1-\partial_{\xi}^2)[(\phi^2-1)k'+(2-b)\phi\phi'k-\lambda k]=0.
\end{equation}

The first-order equation
\begin{equation}\label{41}
 (\phi^2-1)k'+(2-b)\phi\phi'k-\lambda k=0,
\end{equation}
can be solved exactly as follows
\begin{equation}\label{42}
 k(\xi)=
 \begin{cases}
   k_+e^{-\lambda\xi}(1-e^{-2\xi})^{\frac{b-\lambda}{2}-1}, & \xi>0,\\
   k_-e^{-\lambda\xi}(1-e^{2\xi})^{\frac{b+\lambda}{2}-1},& \xi<0,
 \end{cases}
\end{equation}
where $k^+$ and $k^-$ are arbitrary constants. We proceed similarly as in the proof of Lemma \ref{3.1l} for the point spectrum of $L$ without showing all the details. We consider the limits $\xi\rightarrow \pm \infty$ and $\xi\rightarrow 0^{\pm}$ for $\realpart{\lambda}>0$ and find that the corresponding nonzero function
$\upsilon=(1-\partial_{\xi}^2)k$ is in $L^2(\mathbb{R})$ if and only if $k_-=0$ and
\eq{
\frac{b-\realpart{\lambda}}{2}-1>\frac{3}{2}\text{ and }\realpart{\lambda}>0,}{condL} which gives $0<\realpart{\lambda}<b- 5$
 if $b > 5$.

In general, it follows from \eqref{40} that
\begin{equation}\label{43}
 (\phi^2-1)k'+(2-b)\phi\phi'k-\lambda k=
 \begin{cases}
   K_+e^{-\xi}+M_+e^{\xi}, & \xi>0,\\
   K_-e^{\xi}+M_-e^{-\xi},& \xi<0,
 \end{cases}
\end{equation}
where $K_{\pm}$ and $M_{\pm}$ are arbitrary constants. Next, we will show that these
constants must be set to zero so that the general equation \eqref{43} is reduced to \eqref{41}.

Since $k$ and $k'$ are both bounded, we immediately obtain $M_+=M_-=0$ and $K_+=K_-$.
It follows from \eqref{43} with $M_+=M_-=0$ that $k$ and $k'$ are continuous across $\xi=0$
if and only if
\begin{equation*}
\left\{\begin{array}{l}
K_{\pm}=(\pm(b-2)-\lambda)k(0),\\
\mp K_{\pm}=(\pm(b-4)-\lambda)k'(0)+2(2-b)k(0),
\end{array}\right.
\end{equation*}
where we have used $\phi''=\phi, \phi(0^{\pm})=1$ and $\phi'(0^{\pm})=\mp 1$. If $b>5$, this system yields $K^+=K^-=0$, and hence solution \eqref{42} is the only suitable solution of \eqref{39}
such that $k$ and $k'$ are bounded and continuous and $\upsilon\in L^2(\mathbb{R})$.\qed\\

\subsection{Proof of Theorem \ref{2.1}.}
According to Theorem 1 in \cite{GP20}, since $L$ is a compact perturbation of $L_0$, we only need to show that $\sigma_p(L_0)\cap \rho(L)=\emptyset$ and $\sigma_p(L)\cap \rho(L_0)=\emptyset$ to conclude that the two operators have the same spectrum.
By Theorem \ref{3.1} and Lemma \ref{3.1l}, $\sigma_p(L_0)$ consists of the bands for $0<|\realpart{\lambda}|<5-b$ for $b<5$, whereas $\sigma_p(L)$ consists of the same bands and additional eigenvalues given in \eqref{SpecSpecr} if $b\geq 4$ or $b\leq 2$ . However, by Theorem \ref{3.1}, the resolvent set $\rho(L_0)$ consists of the bands $|\realpart{\lambda}|>5-b$. Thus, $\sigma_p(L)\cap \rho(L_0)=\emptyset$. Notice that $\sigma_p(L_0)\subset \sigma_p(L)$,
then we have $\sigma_p(L_0)\cap \rho(L)=\emptyset$.

For the case $b\geq 5$, by Theorem \ref{3.1}, $\sigma_p(L_0)$ is an empty set, hence $\sigma_p(L_0)\cap \rho(L)=\emptyset$. Moreover, $\sigma_p(L)$ is restricted to the imaginary axis and it does not belong to $\rho(L_0)$ in the
bands $|\realpart{\lambda}|>b-5$. Therefore, $\sigma_p(L)\cap \rho(L_0)=\emptyset$.

Note that since $Q$ is a compact operator in $L^2(\mathbb{R})$ by Lemma \ref{2.1Q}, it follows from Theorem 1 in \cite{GP20} that $\sigma(L)=\sigma(L_0)$. Thus, the proof of Theorem \ref{2.1} follows from the the statement of Theorem \ref{3.1}.\qed\\

	\section{Stability analysis on $\HT$}
	\label{HU}
	
	As done in \cite{Lafortune2023} for the Novikov ($b=3$) peakons, we study the spectral and linear stability for the whole Novikov $b$-family  on $\HT$.

 On the space {{$H^1(\R)$,} } we apply the change of variables of Lemma \ref{2.3} and use the linearization \eqref{17} with the condition $\tilde{v}(0,t)=0$. This motivates considering the following subspace of space of $H^1(\R)$
 \eq{
 \widetilde{H}^1:=\left\{v\in H^1(\R) : v(0)=0\right\} 
 }{HTdef}
 and the linear operator given in (\ref{14}), which we rewrite here for convenience as
\beq
\label{Lop3}
 L: =(1-\phi^2)\partial_{\xi}+(4-b)\phi\phi'+Q.
\eeq
On $\widetilde{H}^1$, the domain of $L$ is given by
\beq
\label{Lop-domain4}
{\mbox{Dom}}(L) = \left\{ v\in \widetilde{H}^1: \quad (1-\phi^2) v' \in \widetilde{H}^1 \right\}.
\eeq

We first state and prove the following theorem about the point and residual spectra of $L$ on $\HUZ$.
\begin{theorem}
	\label{L4}
	The linear operator $L$ defined by (\ref{Lop3})--(\ref{Lop-domain4}) on the space $\widetilde{H}^1\subset\HT$, as defined in \eqref{HTdef}, has point spectrum only of $b<3$, in which case it is given by $0<|\realpart{\lambda}| <3-b$.  Furthermore, there is residual spectrum only of $b>3$, in which case it is given by $0<|\realpart{\lambda}| <b-3$.
\end{theorem}
%
\begin{proof}
Without loss of generality, we restrict ourselves to the case where $\realpart{\lambda}\geq 0$. This is due to the symmetry of the eigenvalue problem mentioned in the proof of Theorem \ref{3.1} that implies that the spectrum is symmetric with respect to the imaginary axis.

{\noindent}{\em{Point Spectrum}} Following the argument in the proof of Lemma \ref{3.1l} when considering the point spectrum on $\LT$, we need only to consider the behavior as $\xi\to 0^-$ of the solution of the eigenvalue, which is given in  \eqref{36}.
Instead of the condition \eqref{condL2}, we need here to impose the inequality
$$
2-\frac{\realpart{\lambda}+b}{2}>\frac{1}{2}\text{ and }\realpart{\lambda}>0,
$$
to obtain a necessary and sufficient condition for $v$ to be in $\widetilde{H}^1$. The condition above implies
$
\realpart{\lambda}<3-b,
$
when $b<3$.

{\noindent}{\em{Residual Spectrum}} We first compute the point spectrum of $L^*$. The condition comes out of the argument used
in the proof of Lemma \ref{3.2} where the point spectrum of $L^*$ on $\LT$ is computed. We need only to consider the behavior as $\xi\to 0^+$ of the solution of the eigenvalue, and the condition \eqref{condL} is replaced by
$$
\frac{b-\realpart{\lambda}}{2}-1>-\frac{1}{2}\text{ and }\realpart{\lambda}>0,
$$
for $v$ to be in $\left(\widetilde{H}^1\right)^*$. The condition above implies that $L^*$ has a point spectrum only
if $b>3$ and it is defined by the inequality
$
0<|\realpart{\lambda}|<b-3.
$
Lemma 6.2.6 in \cite{BS18}, that states that
$\sigma_r(L)\subset\sigma_p(L^*)$. Thus in the case
$b\leq 3$, since $L^*$ has no point spectrum, then the residual spectrum of $L$ is empty. Lemma 6.2.6 of
\cite{BS18} also states that in the case where $\sigma_p(L)=\emptyset$ (which is the case when $b\geq 3$ by Theorem \ref{L4}), then the equality
$\sigma_r(L)=\sigma_p(L^*)$ holds, which proves the theorem.
\end{proof}

\begin{remark}
In Theorem \ref{L4}, we do not obtain the whole spectrum. One reason is that $Q$ as defined in \eqref{13} is not compact on $\HUZ$, which makes obtaining the resolvent set more complicated. However, the whole spectrum on $\HUZ$ was shown to
consist of the imaginary axis in \cite{Lafortune2023} for the Novikov equation ($b=3$). In that case, $Q=0$, and thus the analysis is simpler. However, even in that case, it is not trivial.
\end{remark}

The following theorem establishes the spectral stability/instability result on $\HUZ$.
	\begin{theorem}
	\label{S4}
	The peakon solutions of the $b$-Novikov \eqref{1} are spectrally stable on $\HUZ$ only in the Novikov Equation case $b=3$.
\end{theorem} 	
\begin{proof}
The spectral instability result for $b\neq 3$ is a consequence of Theorem \ref{L4}. In the case $b=3$, it was shown in
\cite{Lafortune2023} that the peakons are spectrally and linearly stable solutions of the Novikov equation.
\end{proof}

In order to present our result concerning the linear stability, we first mention that we define linear stability on $\HT$ the same
way as Definition \ref{d:2.1}, with $\LT$ being replaced by $\HT$.

\begin{corollary}
	\label{cor-instab}
	The $b$-Novikov peakons are linearly stable on $\HT$ only in the case $b=3$.
\end{corollary}

\begin{proof}
The linear stability in the case $b=3$ was established in \cite{Lafortune2023}.

The argument to obtain the linear instability result when $b\neq 3$ is almost identical to the part of the proof of Corollary \ref{StabRes1}  that uses Theorem \ref{2.1}  to construct solutions to the linearized equation \eqref{17}, through solutions of the system \eqref{18} and \eqref{19},
that grow exponentially in time. Here we need to use Theorem \ref{L4} and construct the solutions when $b\neq 3$. We mention also that here, in the proof of Corollary \ref{cor-instab}, we use the version of the adjoint $L^*$ computed with respect to the $\HT$ pairing, through the Riesz representation theorem for the dual space. Thus, when constructing the solution of the linearisation using the residual spectrum in the case $b>3$, the eigenfunction of $L^*$ we use is in $\HT$.
\end{proof}

\section{Conclusion.}
\label{Con}
In this paper, we have shown the the peakons of the $b$-Novikov are spectrally and linearly unstable
on $\LT$. This was done by introducing a weak version of the linearized $b$-Novikov and determining completely the spectrum
of the corresponding operator. On the space $\HT$, we show that the peakons are always spectrally unstable except for the integrable Novikov equation ($b=3$). In the $b=3$ case, the peakons were shown to be linearly and spectrally stable in \cite{Lafortune2023}, orbitally stable in $\HT$ with the use of a Lyapunov functional in \cite{Chen21}, and asymptotically stable, also on $\HT$, in \cite{Pal20,Pal21}.  At the linear level, it was shown in \cite{CP20} that the $W^{1,\infty}(\R)$ norm of small perturbations of the Novikov peakons increases exponentially in time as $e^t$.

Consider the $b$-CH equation \eqref{7}, which reduces to the integrable Camassa-Holm equation (CH) when $b=2$ and the integrable Degasperis-Processi equation (DP) when $b=3$. Orbital stability of peakons in $H^1(\mathbb{R})$ was shown
for the CH equation in \cite{CS00,CM01}, while the orbital stability of the  DP peakons on $\LT$ was shown
in \cite{LL09}. Furthermore, the $b$-CH peakons were shown to be linearly and spectrally unstable on $\LT$ for all $b$ in \cite{LP22} and, for $b<3/2$, they were also shown to be linearly and spectrally unstable on the space $H^1(\mathbb{R}^+)\oplus H^1(\mathbb{R}^-)$ in \cite{Char1}. Numerical simulations in \cite{Holm1,Holm2} show that the peakons of the $b$-CH equation are likely to be unstable for $b < 1$ and stable when $b>1$. Given this complicated stability picture, it would be interesting to investigate the linear stability on $\HT$ for the $b$-CH peakons as well.


\noindent{\bf Acknowledgements}

This research was partially by the
Scientific Research Fund of Hunan Provincial Education Department
(No.21A0414). The research of S. Lafortune was supported by a Collaboration Grants for Mathematicians from the Simons Foundation (award \# 420847).


\begin{thebibliography}{99}

\bibitem{MM2013}
Y. Mi, C. Mu, {\sl{On the Cauchy problem for the modified Novikov equation
with peakon solutions}},  J. Differential Equations, {\textbf{254}} (2013) 961--982.
\bibitem{Nov09}
V. Novikov, {\sl{Generalizations of the Camassa--Holm equation}}, J. Phys. A \textbf{42} (2009) 342002.
\bibitem{FF81}
A. Fokas, B. Fuchssteiner, {\sl{Symplectic structures, their B\"{a}cklund transformation and hereditary
symmetries}}, Physica D \textbf{4} (1981) 47--66.
\bibitem{CH93}
R. Camassa, D. Holm, {\sl{An integrable shallow water equation with peaked solitons}}, Phys.
Rev. Lett. \textbf{71} (1993) 1661--1664.
\bibitem{DP99}
A. Degasperis, M. Procesi, {\sl{Symmetry and perturbation theory}}, in: Asymptotic Integrability,
A. Degasperis and G. Gaeta, eds., World Scientific, Singapore, 1999, pp. 23--37.
\bibitem{Con01}
A. Constantin, {\sl{On the scattering problem for the Camassa-Holm equation}}, Proc. Roy. Soc.
London {\bf{457}} (2001), 953--970.
\bibitem{Fo95}
A. S. Fokas, {\sl{On a class of physically important integrable equations}}, Physica D {\bf{87}} (1995),
145--150.
\bibitem{Fu96}
B. Fuchssteiner, {\sl{Some tricks from the symmetry--toolbox for nonlinear equations: Generalizations
of the Camassa-Holm equation}}, Physica D {\bf{95}} (1996), 229--243.
\bibitem{GL13}
G. Gui, Y. Liu, P. Olver, C.Z. Qu, {\sl{Wave breaking and peakons for a modified Camassa-Holm equation}}, Comm. Math. Phys. {\bf{319}} (2013) 731--759.
\bibitem{FT11}
F. Ti\v{g}lay, {\sl{The periodic Cauchy problem for Novikov's equation}}, Int. Math. Res. Not. {\bf{2011}} (2011) 4633-4648.
\bibitem{LP22}
 S. Lafortune, D.E. Pelinovsky, {\sl{Spectral instability of peakons in the b-family of the Camassa-Holm equations}}, SIAM J. Math. Anal. {\bf{54}} (2022) 4572--4590.
\bibitem{Chen16}
R.C. Chen, F. Guo, Y. Liu, C.Z. Qu, {\sl{Analysis on the blow-up solutions to a class of integrable peakon equations}}, J.Funct. Anal. {\bf{270}} (2016) 2343--2374.
\bibitem{OR96}
P.J. Olver, P. Rosenau, {\sl{Tri-Hamiltonian duality between solitons and solitary-wave solutions having compact support}}, Phys. Rev. E {\bf{53}} (1996) 1900--1906.
\bibitem{CS00}
A. Constantin, W. Strauss, {\sl{Stability of peakons}}, Comm. Pure Appl. Math. {\bf{53}} (2000)
603-610.
\bibitem{CM01}
A. Constantin, L. Molinet, {\sl{Orbital stability of solitary waves for a shallow water
equation}}, Physica D {\bf{157}} (2001) 75-89.
\bibitem{JL04}
J. Lenells, {\sl{Stability of periodic peakons}}, Int. Math. Res. Not. {\bf{10}} (2004) 485--499.
\bibitem{JLE04}
J. Lenells, {\sl{A variational approach to the stability of periodic peakons}}, J. Nonlinear Math. Phys. {\bf{11}}
(2004) 151--163.
\bibitem{LL09}
Z. Lin, Y. Liu, {\sl{Stability of peakons for the Degasperis-Procesi equation}}, Commun. Pure Appl. Math. {\bf{62}} (2009) 125--146.
\bibitem{LLQ14}
 X.C. Liu, Y. Liu, C.Z. Qu, {\sl{Stability of peakons for the Novikov equation}}, J. Math.Pures Appl. {\bf{101}} (2014) 172--187.
\bibitem{JP16}
E.R. Johnson, D.E. Pelinovsky, {\sl{Orbital stability of periodic waves in the class of reduced Ostrovsky equations}}, J. Differ. Equ. {\bf{261}} (2016) 3268--3304.
\bibitem{GP19}
A. Geyer,  D. E. Pelinovsky, {\sl{Linear instability and uniqueness of the peaked periodic wave in the
reduced Ostrovsky equation}}, SIAM J. Math. Anal. {\bf{51}} (2019) 1188--1208.
\bibitem{GP20}
A. Geyer, D.E. Pelinovsky, {\sl{Spectral instability of the peaked
periodic wave in the reduced Ostrovsky equation}}, Proc. Amer. Math. Soc. {\bf{148}} (2020) 5109--5125.
\bibitem{NP20}
F. Natali,  D. E. Pelinovsky, {\sl{Instability of $H^1$-stable peakons in the Camassa-Holm equation}}, J.
Diff. Eqs. {\bf{268}} (2020) 7342--7363.
\bibitem{MP20}
A. Madiyeva, D. E. Pelinovsky, {\sl{Growth of perturbations to the peaked periodic waves in the Camassa-Holm equation}}, SIAM J. Math. Anal. {\bf{53}} (2021) 3016--3039.
\bibitem{CP20}
R. M. Chen,  D. E. Pelinovsky, {\sl{$W^{1,\infty}$ instability of $H^1$-stable peakons in the Novikov equation}}, Dynamics of PDE {\bf{18}} (2020) 173--191.
\bibitem{Mol18}
L. Molinet, {\sl{A Liouville property with application to asymptotic stability for the Camassa-Holm equation}}, Arch. Ration. Mech. Anal. {\bf{230}} (2018) 185--230.
\bibitem{Chen21}
R. M. Chen, W. Lian, D. Wang, R. Xu, {\sl{A rigidity property for the Novikov equation and the
asymptotic stability of peakons}}, Arch. Ration. Mech. Anal. {\bf{241}} (2021) 497--533.
\bibitem{Pal20}
J. M. Palacios, {\sl{Asymptotic stability of peakons for the Novikov equation}}, J.~Diff.~Equ.~{\bf{269}} (2020) 7750--7791.
\bibitem{Pal21}
J. M. Palacios, {\sl{Orbital and asymptotic stability of a train of peakons for the Novikov equation}}, Discrete Contin.
Dyn. Syst., {\bf{41}} (2021) 2475--2518.
\bibitem{Lafortune2023} S.~Lafortune, Spectral and linear stability of peakons in the Novikov equation, Stud.~Appl.~Math.~{\bf{152}} (2024) 1404--1424.
\bibitem{Himonas2014} A.~Himonas and C.~Holliman, {\sl{The Cauchy problem for a generalized Camassa–Holm equation}}, Adv. Differ.
Equ. {\bf{19}}, (2014) 161--200.
\bibitem{Himonas2022} A.~Himonas and C.~Holliman, {\sl{Instability and Nonuniqueness for the $b$-Novikov Equation}}, J.~Nonl.~Sc.~{\bf{32}}, (2022): 46.
\bibitem{Renardy}
M.~Renardy and R.~C.~Rogers, %
{\sl{An Introduction to Partial Differential Equations}},
Texts in Applied Mathematics, Springer-Verlag, 2nd edition (2004).
\bibitem{BS18}
T. Buhler, D. A. Salamon, {\sl{Functional Analysis}}, Grad. Stud. Math. {\bf{191}}, AMS, Providence, RI, 2018.
\bibitem{TE12}
G. Teschl, {\sl{Ordinary Differential Equations and Dynamical Systems}}, Grad. Stud. Math. {\bf{140}},
AMS, Providence, RI, 2012.
\bibitem{Char1} E. G. Charalampidis, R. Parker, P. G. Kevrekidis, and S. Lafortune, ``The stability of the $b$-family of peakon equations'', Nonlinearity {\bf{36}} (2023) 1192--1217.
\bibitem{Holm1} D. D. Holm and M. F. Staley, ``Nonintegrability of a fifth-order equation with integrable two-body dynamics", Phys. Lett. A {\bf 308} (2003) 437--444.
\bibitem{Holm2} D. D. Holm and M. F. Staley, ``Wave structure and nonlinear balances in a family of evolutionary PDEs", SIAM J. Appl. Dyn. Syst. {\bf 2} (2003) 323--380.
\end{thebibliography}
\end{document}